\documentclass[11pt]{article}
\usepackage{amssymb}
\usepackage{amsmath}
\usepackage{latexsym}
\usepackage{mathrsfs}
\usepackage{algorithm}
\usepackage{algorithmic}
\usepackage{bm}
\usepackage{color}
\usepackage{verbatim}
\usepackage{url}
\usepackage{graphicx,graphics}
\usepackage{subfigure}

\newtheorem{lemma}{Lemma}
\newtheorem{thm}{Theorem}

\newtheorem{defi}{Definition}
\newtheorem{prop}{Proposition}

\newcommand{\qed}{\mbox{}\hspace*{\fill}\nolinebreak\mbox{$\rule{0.6em}{0.6em}
$}}
\newenvironment{proof}{{\bf Proof\,\,}}{\endproof\par}

\newcounter{spb}
\def \openbox{$\sqcup\llap{$\sqcap$}$}
\def \endproof{\enskip \null \nobreak \hfill \openbox \par}

\newcommand{\bz}{\mathbf 0}

\newcommand{\R}{\mathbb R}

\begin{document}
\title{Quadratic Optimization with Orthogonality Constraints: \\ Explicit {\L}ojasiewicz Exponent and Linear Convergence of Line-Search Methods}
\author{Huikang Liu\thanks{Department of Systems Engineering and Engineering Management, The Chinese University of Hong Kong, Shatin, N.~T., Hong Kong.  E--mail: {\tt hkliu@se.cuhk.edu.hk}} \and Weijie Wu\thanks{Department of Systems Engineering and Engineering Management, The Chinese University of Hong Kong, Shatin, N.~T., Hong Kong.  E--mail: {\tt wwu@se.cuhk.edu.hk}} \and Anthony Man--Cho So\thanks{Department of Systems Engineering and Engineering Management, and, by courtesy, CUHK--BGI Innovation Institute of Trans--omics, The Chinese University of Hong Kong, Shatin, N.~T., Hong Kong.  E--mail: {\tt manchoso@se.cuhk.edu.hk}}}

\date{\today}
\maketitle

\begin{abstract}
A fundamental class of matrix optimization problems that arise in many areas of science and engineering is that of quadratic optimization with orthogonality constraints. Such problems can be solved using line-search methods on the Stiefel manifold, which are known to converge globally under mild conditions. To determine the convergence rate of these methods, we give an explicit estimate of the exponent in a {\L}ojasiewicz inequality for the (non-convex) set of critical points of the aforementioned class of problems.  By combining such an estimate with known arguments, we are able to establish the linear convergence of a large class of line-search methods.  A key step in our proof is to establish a local error bound for the set of critical points, which may be of independent interest.
\end{abstract}

\section{Introduction}
Quadratic optimization problems with orthogonality constraints constitute an important class of matrix optimization problems that have found applications in areas such as combinatorial optimization, data mining, dynamical systems, multivariate statistical analysis, and signal processing, just to mention a few (see, e.g.,~\cite{BMTZ98,M02,AMS08,JNRS10,KCS11,Saad11,S11b,YBGR12}).  A prototypical form of such problems is
\begin{equation}\label{realproblem}
\min_{X \in {\rm St}(m,n)} \,\,\, \left\{ F(X) = {\rm tr}\left( X^TAXB \right) \right\}, 
\end{equation}
where ${\rm St}(m,n) = \left\{ X \in \R^{m\times n} \mid X^TX = I_n \right\}$ (with $m\ge n$ and $I_n$ being the $n\times n$ identity matrix) is the compact Stiefel manifold and $A \in \mathcal{S}^m$, $B \in \mathcal{S}^n$ are given symmetric matrices.  Despite its simplicity, Problem~\eqref{realproblem} already has many applications, a most prominent of which is Principal Component Analysis (PCA).  One of the algorithmic approaches for solving~\eqref{realproblem} is to apply line-search methods on the manifold ${\rm St}(m,n)$.  The update formula of this family of methods takes the form
\begin{equation} \label{eq:it-form}
X_{k+1} = R\left( X_k,\alpha_k\xi_k \right) \quad\mbox{for }k=0,1,\ldots,
\end{equation}
where $\alpha_k \ge 0$ is the step size, $\xi_k$ is a search direction in the tangent space to ${\rm St}(m,n)$ at $X_k$, and $R(X_k,\cdot)$ is a so-called \emph{retraction} that maps a vector in the tangent space to ${\rm St}(m,n)$ at $X_k$ into a point on ${\rm St}(m,n)$.  In particular, the iterates produced by~\eqref{eq:it-form} are all feasible for Problem~\eqref{realproblem}.  Naturally, the choice of step sizes, search directions and the retraction will affect the convergence and efficiency of the resulting method.  For the general problem of optimizing a smooth function over the Stiefel manifold (which includes Problem~\eqref{realproblem} as a special case), various choices have been proposed over the years, and the convergence properties of the resulting methods are relatively well understood; see, e.g.,~\cite{AEK08,AMS08,AM12,WY13,JD14}.  However, very little is known about the convergence rates of these methods, even when they are applied to the much more structured problem~\eqref{realproblem}.  Part of the difficulty is due to the fact that optimization problems over the Stiefel manifold are non-convex in general.  This implies that much of the existing analysis machinery, which heavily exploits convexity, cannot be applied to such problems.  Currently, convergence rates of line-search methods for solving Problem~\eqref{realproblem} are established only under quite restrictive conditions.  For instance, Absil et al.~\cite[Theorem 4.6.3]{AMS08} showed that when $n=1$ and $B=I_n=1$ (and hence Problem~\eqref{realproblem} corresponds to minimizing the Rayleigh quotient on the unit sphere in $\R^m$), a certain line-search method will converge linearly to an eigenvector corresponding to the smallest eigenvalue $\lambda$ of $A$, provided that $\lambda$ is simple.  More recently, Shamir~\cite{S15} developed a stochastic line-search method for Problem~\eqref{realproblem} when $n=1$, $B=I_n=1$, and $A$ is negative semidefinite.  He showed that if the smallest eigenvalue $\lambda$ is simple and certain boundedness assumptions hold, then his proposed method converges linearly to an eigenvector corresponding to $\lambda$.  However, it is not clear how to extend the above results to handle the case where $n>1$ and/or the multiplicity of $\lambda$ is greater than one.  On another front, Smith~\cite{S94} showed that when used to optimize a smooth function over a Riemannian manifold, the method of steepest descent will converge linearly to a critical point if the function is \emph{strongly convex on the manifold}.  However, such a notion of convexity is much stronger than that on the Euclidean space.  In particular, it is known that every smooth function that is convex on a compact Riemannian manifold (such as the Stiefel manifold) is constant~\cite{BO69}.  Therefore, one cannot hope to obtain linear convergence results for Problem~\eqref{realproblem} using the convexity-based approach in~\cite{S94}.  Recently, there have been some endeavors to analyze the convergence rates of line-search methods for solving optimization problems over embedded submanifolds using the so-called \emph{{\L}ojasiewicz inequality}; see, e.g.,~\cite{AMA05,MN13,SU15}.  Although such an approach is extremely powerful, it has a severe limitation; namely, the exponent in the {\L}ojasiewicz inequality is often hard to determine explicitly.  Without the knowledge of such exponent, one cannot determine the exact rate of convergence of a given method.  As it turns out, the {\L}ojasiewicz exponent for general polynomial systems is known (see, e.g.,~\cite{LMP14}) and can in principle be applied to Problem~\eqref{realproblem}.  However, the exponent depends on the dimensions of the problem and leads only to very weak convergence rate results.

In view of the above discussion, our main contribution of this paper is to give a significantly sharper estimate of the {\L}ojasiewicz exponent for the non-convex problem~\eqref{realproblem}.  In particular, it is independent of the dimensions of the problem.  We achieve this by establishing a local Lipschitzian error bound for the (non-convex) set of critical points of Problem~\eqref{realproblem}, which may be of independent interest.  By combining our estimate of the {\L}ojasiewicz exponent with a well-established analysis framework in the literature~\cite{SU15}, we conclude that a host of line-search methods for solving Problem~\eqref{realproblem} converge linearly to a critical point.  It should be noted that our convergence result does not require any restriction on the eigenvalues of $A$ and $B$.  Thus, it is qualitatively different from those in~\cite{AMS08,S15}.  Moreover, although our work is similar in spirit as~\cite{LT93,T10,S13,ZZS15}, there is a crucial difference: While the latter deals exclusively with convex optimization problems, the former considers an optimization problem in which neither the objective function nor the constraint is convex.

Besides the notations introduced earlier, we shall use $\mathcal{O}^n$ to denote the set of $n\times n$ orthogonal matrices (in particular, we have $\mathcal{O}^n={\rm St}(n,n)$); ${\rm Diag}(x_1,\ldots,x_n)$ to denote the diagonal matrix with $x_1,\ldots,x_n$ on the diagonal; ${\rm BlkDiag}(A_1,\ldots,A_n)$ to denote the block diagonal matrix whose diagonal blocks are $A_1,\ldots,A_n$.  Given a matrix $Y \in \R^{m\times n}$ and a non-empty closed set $\mathcal{X} \subset \R^{m\times n}$, we shall use ${\rm dist}(Y,\mathcal{X})$ to denote the distance of $Y$ to $\mathcal{X}$; i.e., ${\rm dist}(Y,\mathcal{X})=\min_{X\in\mathcal{X}}\|X-Y\|_F$.  Other notations are standard.

\section{Background}
\subsection{First-Order Optimality Condition and Descent Directions}
To begin, let us introduce some basic definitions and concepts.  We view ${\rm St}(m,n)$ as an embedded submanifold of $\R^{m\times n}$ with the inherited Riemannian metric $\langle \cdot,\cdot \rangle$ given by $\langle X,Y \rangle = {\rm tr}\left( X^TY \right)$.  For any $X\in{\rm St}(m,n)$, the tangent space to ${\rm St}(m,n)$ at $X$ is given by $T(X) = \left\{ Y \in \R^{m\times n} \mid X^TY + Y^TX = \bz \right\}$.  The gradient of $F(X)={\rm tr}\left( X^TAXB \right)$ is $\nabla F(X)=2AXB$, and its orthogonal projection onto $T(X)$ is given by
\begin{align*}
{\rm grad}\,F(X) &= \left( I_m - XX^T \right)\nabla F(X) + \frac{1}{2}X\left( X^T\nabla F(X) - \nabla F(X)^TX \right) \\
&= 2AXB - XX^TAXB - XBX^TAX.
\end{align*}
Let $\mathcal{X} = \left\{ X \in {\rm St}(m,n) \mid {\rm grad}\,F(X) = \bz \right\}$ be the set of critical points of Problem~\eqref{realproblem}.  The following proposition gives a characterization of $\mathcal{X}$:
\begin{prop} \label{prop:grad-eqv}
Let $X\in{\rm St}(m,n)$ be given.  Then, the following are equivalent:
\begin{enumerate}
\item[(i)] ${\rm grad}\,F(X) = \bz$.

\item[(ii)] $\nabla F(X) - X\nabla F(X)^TX = \bz$.

\item[(iii)] For any $\rho>0$, $D_\rho(X) = \nabla F(X) - X\left( 2\rho\nabla F(X)^TX + (1-2\rho)X^T\nabla F(X) \right) = \bz$.
\end{enumerate}
\end{prop}
\begin{proof}
The equivalence between (ii) and (iii) is established in~\cite[Lemma 2.1]{JD14}.  To prove the equivalence between (i) and (ii), observe that
\begin{align*}
{\rm grad}\,F(X) &= \left( I_m - \frac{1}{2}XX^T \right) \nabla F(X) - \frac{1}{2}X\nabla F(X)^TX \\
&= \left( I_m - \frac{1}{2}XX^T \right) \left( \nabla F(X) - X\nabla F(X)^TX \right). 
\end{align*}
Now, it remains to note that $I_m-(1/2)XX^T$ is invertible.
\end{proof}
It is easy to verify that $D_\rho(X) \in T(X)$ for any $\rho>0$.  Moreover, as shown in~\cite[Lemma 3.1]{JD14}, $-D_\rho(X)$ is a descent direction at $X\in{\rm St}(m,n)$ for any $\rho>0$.  Hence, in the sequel, we shall focus on line-search methods that use $-D_\rho(\cdot)$ as the search direction.

\subsection{Retraction}
Another ingredient in line-search methods for optimizing over ${\rm St}(m,n)$ is a retraction:
\begin{defi}{(Retraction)}
A map $R: \bigcup_{X\in {\rm St}(m,n)} \{X\}\times T(X) \rightarrow {\rm St}(m,n)$ will be called a retraction, if for any fixed $X\in{\rm St}(m,n)$ and $\xi\in T(X)$ it holds that $\xi\mapsto R(X,\xi)$ is continuous on $T(X)$, and for all $X\in {\rm St}(m,n)$,
\begin{equation}\label{retraction}
\lim_{T(X)\ni\xi\rightarrow 0} \frac{\left\|R(X, \xi)-(X+\xi)\right\|_F}{\left\|\xi \right\|_F}=0.
\end{equation}
\end{defi}
Various smooth retractions on the Stiefel manifold have been proposed in the literature.  These include the polar decomposition-based retraction, the QR-decomposition-based retraction, the Cayley transform, and the Riemannian exponential mapping.  We refer the reader to~\cite{AMS08,KFT13} for details of these retractions.  In Section~\ref{sec:exp}, we shall conduct numerical experiments with these four retractions.

\subsection{Step Sizes}
To complete the specification of a line-search method, it remains to choose the step sizes.  This is done in the following:
\begin{defi}{(Armijo Point)}
Let $\gamma>0$, $\beta, c\in(0, 1)$ be given constants. The number
\begin{equation}\label{armijo}
\alpha = \max\left\{ \beta^n\gamma \mid n\ge0, F\left( R\left(X,-\beta^n\gamma D_\rho(X)\right) \right) - F(X) \leq -c\beta^n\gamma \nabla F(X)^TD_\rho(X) \right\}
\end{equation}
is called the Armijo point at $X\in{\rm St}(m,n)$ with parameters $(\gamma,\beta,c)$.
\end{defi}
Since the smooth retraction~\eqref{retraction} is  a first-order approximation, the left hand side approximate the first-order derivative along $ -\beta^n\gamma D_\rho$ when $m$ is large enough.  Consequently, the Armijo point exists.  We refer the reader to~\cite{SU15} for details.

We summarize the line-search method in Algorithm~\ref{alg:1}.
\begin{algorithm}[htb]
\caption{Line-Search Method on the Stiefel manifold} 
\label{alg:1} 
\begin{algorithmic}[1]
\REQUIRE  Select $X_0\in {\rm St}(m,n)$, $\rho>0$, $\beta, c\in(0,1)$.
\FOR{$k=0,1,2,\ldots $}
\STATE  Calculate the descent direction $-D_\rho(X_k)$ at $X_k$.
\STATE  Choose $\bar{\beta_k}\geq 1$ and find the Armijo point $\alpha_k$ at $X_k$ with parameters $(\bar{\beta_k}, \beta, c)$.
\STATE   Set $X_{k+1}=R\left( X_k,-\alpha_kD_{\rho}(X_k) \right)$.
\ENDFOR
\end{algorithmic}
\end{algorithm}

\subsection{Convergence Analysis Framework for the Line-Search Method}
To analyze the convergence properties of Algorithm~\ref{alg:1}, we adopt the framework introduced in~\cite{SU15}.  It has been shown in~\cite[Corollary 2.9]{SU15} that Algorithm~\ref{alg:1} has the following properties:
\begin{itemize}
\item[$\bullet$] (Primary Descent) There exists a constant $\sigma > 0$ such that for all $k$ large enough,
$$
F(X_{k+1})-F(X_k)\leq -\sigma\left\|D_{\rho}(X_k)\right\|_F\left\|X_{k+1}-X_k\right\|_F.
$$
\item[$\bullet$] (Stationarity) For all $k$ large enough, 
$$
\left\|D_{\rho}(X_k)\right\|_F=0\quad \Longrightarrow\quad X_{k+1}=X_k.
$$
\end{itemize}
Moreover, we show in the appendix that Algorithm~\ref{alg:1} has the following property:
\begin{prop} \label{prop:asy-safe}
(Asymptotic Small Step Size Safeguard) There exists a constant $\kappa > 0$ such that for all $k$ large enough,
\begin{equation} \label{eq:asy-safe}
\left\|X_{k+1}-X_k\right\|_F\geq \kappa\left\|D_{\rho}(X_k)\right\|_F.
\end{equation}
\end{prop}
Thus, by~\cite[Theorem 2.3]{SU15}, in order to establish the linear convergence of Algorithm~\ref{alg:1} to a critical point of Problem~\eqref{realproblem}, it remains to prove the following theorem:
\begin{thm} \label{thm:loj-ineq}
({\L}ojasiewicz Inequality for Quadratic Optimization with Orthogonality Constraints) There exist constants $\delta,\eta > 0$ such that for all $X \in {\rm St}(m,n)$ and $X^* \in \mathcal{X}$ with $\left\|X-X^*\right\|_F \le \delta$,
$$
|F(X)-F(X^*)|^{1/2}\leq \eta\left\|D_{\rho}(X)\right\|_F.
$$
\end{thm}
The proof of Theorem~\ref{thm:loj-ineq} is based on the following two results:
\begin{thm}\label{theorem:A5}
(Local Error Bound for Quadratic Optimization with Orthogonality Constraints) There exist constants $\delta,\eta>0$ such that 
$$
{\rm dist}(X,\mathcal{X}) \le \eta \| D_\rho(X) \|_F \quad\mbox{whenever }  X\in{\rm St}(m,n) \mbox{ and } {\rm dist}(X,\mathcal{X}) \le \delta.
$$
\end{thm}
We defer the proof of Theorem~\ref{theorem:A5} to Section~\ref{sec:pf}.
\begin{prop} (2-H\"{o}lder Continuity of $F$) \label{prop:2-hold}
There exists a constant $\eta>0$ such that for all $X\in{\rm St}(m,n)$ and $X^*\in\mathcal{X}$,
$$ |F(X)-F(X^*)| \le \eta\|X-X^*\|_F^2. $$
\end{prop}
\begin{proof}
Observe that $F$, when viewed as a function on $R^{m\times n}$, is continuously differentiable with Lipschitz continuous gradient.  Thus, we have
\begin{equation} \label{eq:grad-lip}
\left| F(X)-F(X^*)-\langle\nabla F(X^*), X-X^*\rangle \right| \leq \frac{L}{2}\left\|X-X^*\right\|_F^2,
\end{equation}
where $L>0$ is the Lipschitz constant of $\nabla F$; see, e.g.,~\cite{N04}.  Now, by Proposition~\ref{prop:grad-eqv}, we have $\nabla F(X^*)=X^*\nabla F(X^*)^TX^*$.  This implies that
\begin{equation}\label{eq:part1}
\langle\nabla F(X^*), X-X^*\rangle= \left\langle X^*\nabla F(X^*)^TX^*, X-X^*\right\rangle=\left\langle\nabla F(X^*)^TX^*, (X^*)^TX-I_n\right\rangle.
\end{equation}
On the other hand, 
\begin{align}
\left\langle\nabla F(X^*)^TX^*, I_n-X^TX^*\right\rangle &= \left\langle (X^*)^T\nabla F(X^*), (X^*)^TX^*-X^TX^* \right\rangle \nonumber \\
&= \left\langle X^*\nabla F(X^*)^TX^*, X^*-X \right\rangle \nonumber \\
&= -\langle \nabla F(X^*), X-X^*\rangle. \label{eq:part2}
\end{align}
Upon adding~\eqref{eq:part1} and~\eqref{eq:part2} and using the fact that $(X-X^*)^T(X-X^*)=2I_n-(X^*)^TX-X^TX^*$, we obtain
$$
2\langle \nabla F(X^*), X-X^*\rangle=-\left\langle\nabla F(X^*)^TX^*,(X-X^*)^T(X-X^*)\right\rangle,
$$
or equivalently,
$$
|\langle \nabla F(X^*), X-X^*\rangle|\leq\frac{1}{2}\left\|\nabla F(X^*)^TX^*\right\|_F\|X-X^*\|_F^2.
$$
This, together with~\eqref{eq:grad-lip}, yields the desired inequality with $\eta=\left(L+\left\|\nabla F(X^*)^TX^*\right\|_F\right)/2$.
\end{proof}
To complete the proof of Theorem~\ref{thm:loj-ineq}, let $X\in{\rm St}(m,n)$ and $X^*\in\mathcal{X}$ be such that ${\rm dist}(X,\mathcal{X}) = \|X-X^*\|_F \le \delta$, where $\delta>0$ is given by Theorem~\ref{theorem:A5}.  Then, by Proposition~\ref{prop:2-hold}, we obtain
$$ |F(X)-F(X^*)| \le \eta_1\|X-X^*\|_F^2 = \eta_1\cdot{\rm dist}(X,\mathcal{X})^2 \le \eta_1\eta_2 \|D_\rho(X)\|_F^2 $$
for some constants $\eta_1,\eta_2>0$, as desired.

\section{Proof of Theorem~\ref{theorem:A5}} \label{sec:pf}
We now prove Theorem~\ref{theorem:A5}, which is the main result of this paper.  The proof can be divided into four steps.
\subsection{Preliminary Observations}
Let $A = U_A\Sigma_A U_A^T$ and $B = U_B\Sigma_B U_B^T$ be spectral decompositions of $A$ and $B$, respectively.  It is straightforward to verify that $\mbox{tr}\left( X^TAXB \right) = \mbox{tr}\left( \bar{X}^T\Sigma_A\bar{X}\Sigma_B \right)$, where $\bar{X} = U_A^TXU_B \in {\rm St}(m,n)$.  Thus, we may assume without loss of generality that
$$ A = \mbox{Diag}(a_1,\ldots,a_m) \in \mathcal{S}^m \quad\mbox{and}\quad B = \mbox{Diag}(b_1,\ldots,b_n) \in \mathcal{S}^n, $$
where $a_1 \ge a_2 \ge \cdots \ge a_m$ and $b_1 \ge b_2 \ge \cdots \ge b_n$.  By Proposition~\ref{prop:grad-eqv}, we can write
\begin{equation} \label{eq:opt-set-sim}
\mathcal{X} = \left\{ X \in {\rm St}(m,n) \mid AXB - XBX^TAX = \bz \right\}.
\end{equation}
Now, it can be verified that
$$
D_\rho(X) = \left( I_m-(1-2\rho)XX^T \right)\left( \nabla F(X)-X\nabla F(X)^TX \right).
$$
Since $\rho>0$, we see that $I_m-(1-2\rho)XX^T$ is invertible and
$$
\left\| \nabla F(X)-X\nabla F(X)^TX \right\|_F \le \left\|\left( I_m-(1-2\rho)XX^T\right)^{-1}\right\| \cdot \left\|D_\rho(X)\right\|_F\le \frac{1}{2\rho}\left\|D_\rho(X)\right\|_F.
$$
In particular, in order to prove Theorem~\ref{theorem:A5}, it suffices to prove the following:

\medskip
\noindent{\bf Theorem~\ref{theorem:A5}'.} {\it There exist constants $\delta,\eta>0$ such that 
$${\rm dist}(X,\mathcal{X}) \le \eta \left\| AXB - XBX^TAX \right\|_F \quad\mbox{whenever }  X\in{\rm St}(m,n) \mbox{ and } {\rm dist}(X,\mathcal{X}) \le \delta. $$}

\subsection{Characterizing the Set of Critical Points when $B$ has Full Rank}
Consider first the case where $B$ has full rank; i.e., $b_i\not=0$ for $i=1,\ldots,n$.  Let $n_A$ and $n_B$ be the number of distinct eigenvalues of $A$ and $B$, respectively.  Then, there exist indices $s_0,s_1,\ldots,s_{n_A}$ and $t_0,t_1,\ldots,t_{n_B}$ such that $0=s_0 < s_1 < \cdots < s_{n_A} = m$ and $0=t_0 < t_1 < \cdots < t_{n_B} = n$, and
\begin{eqnarray*}
& & a_{s_0+1} = \cdots = a_{s_1} > a_{s_1+1} = \cdots = a_{s_2} > \cdots > a_{s_{n_A-1}+1} = \cdots = a_{s_{n_A}}, \\
& & b_{t_0+1} = \cdots = b_{t_1} > b_{t_1+1} = \cdots = b_{t_2} > \cdots > b_{t_{n_B-1}+1} = \cdots = b_{t_{n_B}}.
\end{eqnarray*}
Let $U_1,\ldots,U_{n_A}$ and $V_1,\ldots,V_{n_B}$ be the eigenspaces of $A$ and $B$, respectively.  Note that $\mbox{dim}(U_i) = s_i-s_{i-1}$ for $i=1,\ldots,n_A$ and $\mbox{dim}(V_j)=t_j-t_{j-1}$ for $j=1,\ldots,n_B$.  Furthermore, let
$$\mathcal{H} = \left\{ (h_1,\ldots,h_{n_A}) \,\left|\, \sum_{i=1}^{n_A} h_i = n, \,\, h_i \in \{0,1,\ldots,s_i-s_{i-1}\} \,\mbox{ for }\, i=1,\ldots,n_A \right. \right\} $$
and $\{e_i\}_{i=1}^m$ be the standard basis of $\R^m$.  Given any $h=(h_1,\ldots,h_{n_A}) \in \mathcal{H}$, define
\begin{align}
E_i(h) &= [ e_{s_{i-1}+1} \, \cdots \, e_{s_{i-1}+h_i} ] \in \R^{m\times h_i} & \mbox{for } i=1,\ldots,n_A, \nonumber \\
E(h) &= [ E_1(h) \, \cdots \, E_{n_A}(h) ] \in \R^{m\times n}. \label{eq:Eh}
\end{align}
We then have the following characterization of the set $\mathcal{X}$ of critical points of Problem~\eqref{realproblem}, whose proof can be found in the appendix:
\begin{prop}\label{the:limit}
The following holds:
\begin{align}
\mathcal{X} &= \left\{ X \in {\rm St}(m,n) \mid X = {\rm BlkDiag}(P_1,\ldots,P_{n_A}) \cdot E(h) \cdot {\rm BlkDiag}(Q_1,\ldots,Q_{n_B}) \right. \nonumber \\
&\quad\quad \left. \mbox{for some }\, P_i \in \mathcal{O}^{s_i-s_{i-1}} \, (i=1,\ldots,n_A), \, Q_j \in \mathcal{O}^{t_j-t_{j-1}} \, (j=1,\ldots,n_B), \mbox{ and } h\in\mathcal{H} \right\}. \label{eq:opt-set-char}
\end{align}
\end{prop}
{\it Remarks.} (i) Essentially, Proposition~\ref{the:limit} states that every $X\in\mathcal{X}$ can be factorized as $X=PQ$, where $P \in{\rm St}(m,n)$ and $Q\in\mathcal{O}^n$, and the columns of $P$ (resp.~$Q$) are the eigenvectors of $A$ (resp.~$B$).  Indeed, observe that for $i=1,\ldots,n_A$, the $(s_{i-1}+1)$-st to $s_i$-th columns of ${\rm BlkDiag}(P_1,\ldots,P_{n_A})$ form an orthonormal basis of $U_i$.  Similarly, for $j=1,\ldots,n_B$, the $(t_{j-1}+1)$-st to $t_j$-th columns of ${\rm BlkDiag}(Q_1,\ldots,Q_{n_B})$ form an orthonormal basis of $V_j$.  To specify which $n$ of the $m$ eigenvectors of $A$ are chosen to form $P$, we use the matrix $E(h)$, where $h=(h_1,\ldots,h_{n_A}) \in \mathcal{H}$ and $h_i$ is the number of eigenvectors chosen from the eigenspace $U_i$.

(ii) A result similar to Proposition~\ref{the:limit} has appeared in~\cite[Section 4.8.2]{AMS08}.  However, the proof therein contains a small gap.  Specifically, from the properties that $B$ is diagonal and commutes with $X^TAX$, it is claimed in~\cite[Section 4.8.2]{AMS08} that $X^TAX$ is also diagonal.  However, this is not true unless the diagonal entries of $B$ are all distinct.

Proposition~\ref{the:limit} suggests that we can partition $\mathcal{X}$ into disjoint subsets $\{\mathcal{X}_h\}_{h\in\mathcal{H}}$, where 
\begin{align*}
\mathcal{X}_h &= \left\{ X \in {\rm St}(m,n) \mid X = {\rm BlkDiag}(P_1,\ldots,P_{n_A})  \cdot E(h) \cdot {\rm BlkDiag}(Q_1,\ldots,Q_{n_B}) \right. \nonumber \\
&\quad\quad \left. \mbox{for some }\, P_i \in \mathcal{O}^{s_i-s_{i-1}} \, (i=1,\ldots,n_A), \, Q_j \in \mathcal{O}^{t_j-t_{j-1}} \, (j=1,\ldots,n_B) \right\}.
\end{align*}
Consequently, in order to prove Theorem~\ref{theorem:A5}', it suffices to bound ${\rm dist}(X,\mathcal{X}_h)$ for any $X\in{\rm St}(m,n)$ and $h\in\mathcal{H}$.

\subsection{Estimating the Distance to the Set of Critical Points}
Let $X \in {\rm St}(m,n)$ and $h=(h_1,\ldots,h_{n_A})\in\mathcal{H}$ be arbitrary.  By definition, 
\begin{align}
{\rm dist}(X,\mathcal{X}_h) &= \min\left\{ \left\| X - {\rm BlkDiag}\left( P_1,\ldots,P_{n_A} \right) \cdot E(h) \cdot {\rm BlkDiag}\left( Q_1,\ldots,Q_{n_B} \right) \right\|_F \mid \right. \nonumber \\
&\qquad\quad\,\,\, \left. P_i \in \mathcal{O}^{s_i-s_{i-1}} \mbox{ for } i=1,\ldots,n_A; \, Q_j \in \mathcal{O}^{t_j-t_{j-1}} \mbox{ for } j=1,\ldots,n_B \right\}. \label{dist}
\end{align}
Let $\left( P_1^*,\ldots,P_{n_A}^*,Q_1^*,\ldots,Q_{n_B}^* \right)$ be an optimal solution to~\eqref{dist}.  Upon letting 
$$ P^* = {\rm BlkDiag}\left( P_1^*,\ldots,P_{n_A}^* \right) \in \mathcal{O}^m, \quad Q^* = {\rm BlkDiag}\left( Q_1^*,\ldots,Q_{n_B}^* \right) \in \mathcal{O}^n, $$
and $\bar{X} = (P^*)^TX(Q^*)^T$, it is clear that ${\rm dist}^2(X,\mathcal{X}_h) = \left\| \bar{X} - E(h) \right\|_F^2$.  To bound this quantity, consider the decompositions 
\begin{equation} \label{eq:B-decomp}
\bar{X} = \left[ \bar{X}_1 \, \cdots \, \bar{X}_{n_B} \right] \quad\mbox{and}\quad E(h) = \left[ \bar{E}_1(h) \,\cdots\, \bar{E}_{n_B}(h) \right],
\end{equation}
where $\bar{X}_j,\bar{E}_j(h) \in \R^{m\times (t_j-t_{j-1})}$.  We then have the following result, whose proof can be found in the appendix:
\begin{prop} \label{prop:sq-dist-bd}
For $j=1,\ldots,n_B$ and $k=1,\ldots,m$, denote the $k$-th row of $\bar{X}_j$ and $\bar{E}_j(h)$ by $\left[ \bar{X}_j \right]_k$ and $\left[ \bar{E}_j(h) \right]_k$, respectively.  Suppose that ${\rm dist}(X,\mathcal{X}_h) \le \delta$ for some $\delta \in (0,1)$.  Then, 
$$ {\rm dist}^2(X,\mathcal{X}_h) = \sum_{j=1}^{n_B} \sum_{k\in\mathcal{I}_j} \Theta\left( \left\| \left[\bar{X}_j\right]_k \right\|_2^2 \right), $$
where $\mathcal{I}_j = \left\{ k \in \{1,\ldots,m\}: \left[ \bar{E}_j(h) \right]_k = \bz \right\}$. 
\end{prop}
To establish the desired error bound, we need to link $\left\| AXB - XBX^TAX \right\|_F$ to the bound on ${\rm dist}^2(X,\mathcal{X}_h)$ in Proposition~\ref{prop:sq-dist-bd}.  This is achieved in two steps.  First, we prove the following result:
\begin{prop} \label{prop:grad-prelim-bd}
Consider the decomposition of $\bar{X}$ in~\eqref{eq:B-decomp}.  Then, 
$$ \left\| AXB - XBX^TAX \right\|_F^2 = \Omega\left( \sum_{j=1}^{n_B} \left\| A\bar{X}_j - \bar{X}_j\bar{X}_j^TA\bar{X}_j \right\|_F^2 \right). $$
\end{prop}
In view of Proposition~\ref{prop:grad-prelim-bd}, we then proceed to prove the following bound:
\begin{prop} \label{prop:grad-term-bd}
Suppose that ${\rm dist}(X,\mathcal{X}_h) \le \delta$ for some $\delta \in (0,1)$.  Then, 
$$ \sum_{j=1}^{n_B} \left\| A\bar{X}_j - \bar{X}_j\bar{X}_j^TA\bar{X}_j \right\|_F^2 = \sum_{j=1}^{n_B}\sum_{k\in\mathcal{I}_j} \Omega\left( \left\|  \left[\bar{X}_j\right]_k \right\|_2^2 \right). $$
\end{prop}
The proofs of Propositions~\ref{prop:grad-prelim-bd} and~\ref{prop:grad-term-bd} can be found in the appendix.  Now, observe that whenever $X\in{\rm St}(m,n)$ and ${\rm dist}(X,\mathcal{X})\le\delta$, then there exists an $h\in\mathcal{H}$ such that ${\rm dist}(X,\mathcal{X}_h)\le\delta$.  Hence, by combining Propositions~\ref{prop:sq-dist-bd},~\ref{prop:grad-prelim-bd}, and~\ref{prop:grad-term-bd}, we obtain Theorem~\ref{theorem:A5}'.

\subsection{Removing the Full Rank Assumption on $B$}
Consider now the case where $B$ does not have full rank.  Without loss of generality, we assume that $B = {\rm BlkDiag}(\bar{B},\bz)$, where $\bar{B}={\rm Diag}(b_1,\ldots,b_p)\in\mathcal{S}^p$ has full rank.  Then, using~\eqref{eq:opt-set-sim}, it can be shown that
$$ \mathcal{X} = \left\{ X = [X_1\,\,X_2] \in {\rm St}(m,n) \mid X_1\in\R^{m\times p}, X_2\in\R^{m\times(n-p)}, AX_1\bar{B} - X_1\bar{B}X_1^TAX_1 = \bz \right\}. $$
It follows that for any $X = [X_1\,\,X_2] \in{\rm St}(m,n)$ with $ X_1\in\R^{m\times p}$ and $X_2\in\R^{m\times(n-p)}$, we have ${\rm dist}(X,\mathcal{X}) = {\rm dist}(X_1,\bar{\mathcal{X}})$, where
$$ \bar{\mathcal{X}} = \left\{ X \in {\rm St}(m,p) \mid AX\bar{B} - X\bar{B}X^TAX = \bz \right\}. $$
By our previous result, there exist constants $\delta,\eta>0$ such that 
$$ {\rm dist}(X_1,\bar{\mathcal{X}}) \le \eta\left\| AX\bar{B} - X\bar{B}X^TAX \right\|_F $$
whenever $X_1\in{\rm St}(m,p)$ and ${\rm dist}(X_1,\bar{\mathcal{X}}) \le \delta$.  To complete the proof, it remains to observe that
\begin{align*}
\left\| AXB - XBX^TAX \right\|_F^2 &= \left\| AX_1\bar{B} - X_1\bar{B}X_1^TAX_1 \right\|_F^2 + \left\|X_1\bar{B}X_1^TAX_2\right\|_F^2 \\
&= \left\| AX_1\bar{B} - X_1\bar{B}X_1^TAX_1 \right\|_F^2 + \left\| X_2^T\left( AX_1\bar{B} - X_1\bar{B}X_1^TAX_1 \right)X_1^T \right\|_F^2 \\
&= \Theta\left( \left\| AX_1\bar{B} - X_1\bar{B}X_1^TAX_1 \right\|_F^2 \right).
\end{align*}

\section{Numerical Experiments} \label{sec:exp}
In this section, we perform numerical experiments to investigate the convergence rate of the retracted line-search algorithm for problem~\eqref{realproblem} on synthetic datasets. As we shall see, the results consistency with the theoretical analysis in previous sections. In particular, we consider the four retractions mentioned above.

First, we generate our diagonal matrices $A\in \mathcal{S}^m$ and $B\in \mathcal{S}^n$, whose diagonal elements are sampled randomly from the uniform distribution. The starting point $X_0$ is chosen from the uniform distribution and get the orthonormal basis for the range of $X_0$ to keep the feasibility. In the setting of Armijo point, we fix $\gamma=1, \beta=0.5$ and  $c=0.001$. We stop the algorithm when $F(X_k)-F(X_{k+1})<10^{-8}$.

In practical computations, the orthogonality constraint may be violated after several iterations, which is mainly due to numerical errors incurred in the multiplication. In the numerical experiments, we follow the technique introduced in~\cite{JD14} and use $(X^TX)^{-1}$ to control feasibility error.

Figure~\eqref{fig:thin} illustrates the convergence performance of the four retractions with the relative ``Thin'' matrix: \eqref{fig:1} $m=20, n=10$, \eqref{fig:2} $m=30, n=10$, \eqref{fig:3} $m=100, n=10$. Figure~\eqref{fig:fat} illustrates the convergence performance with the relative ``Fat'' matrix: \eqref{fig:4} $m=20, n=15$, \eqref{fig:5} $m=50, n=40$, \eqref{fig:6} $m=100, n=80$. It can be seen that as long as the iterates are close enough to the optimal set, both the objective values and the solutions converge linearly.
\begin{figure}[h]
\centering
\subfigure[]{
\includegraphics[width=0.3\textwidth]{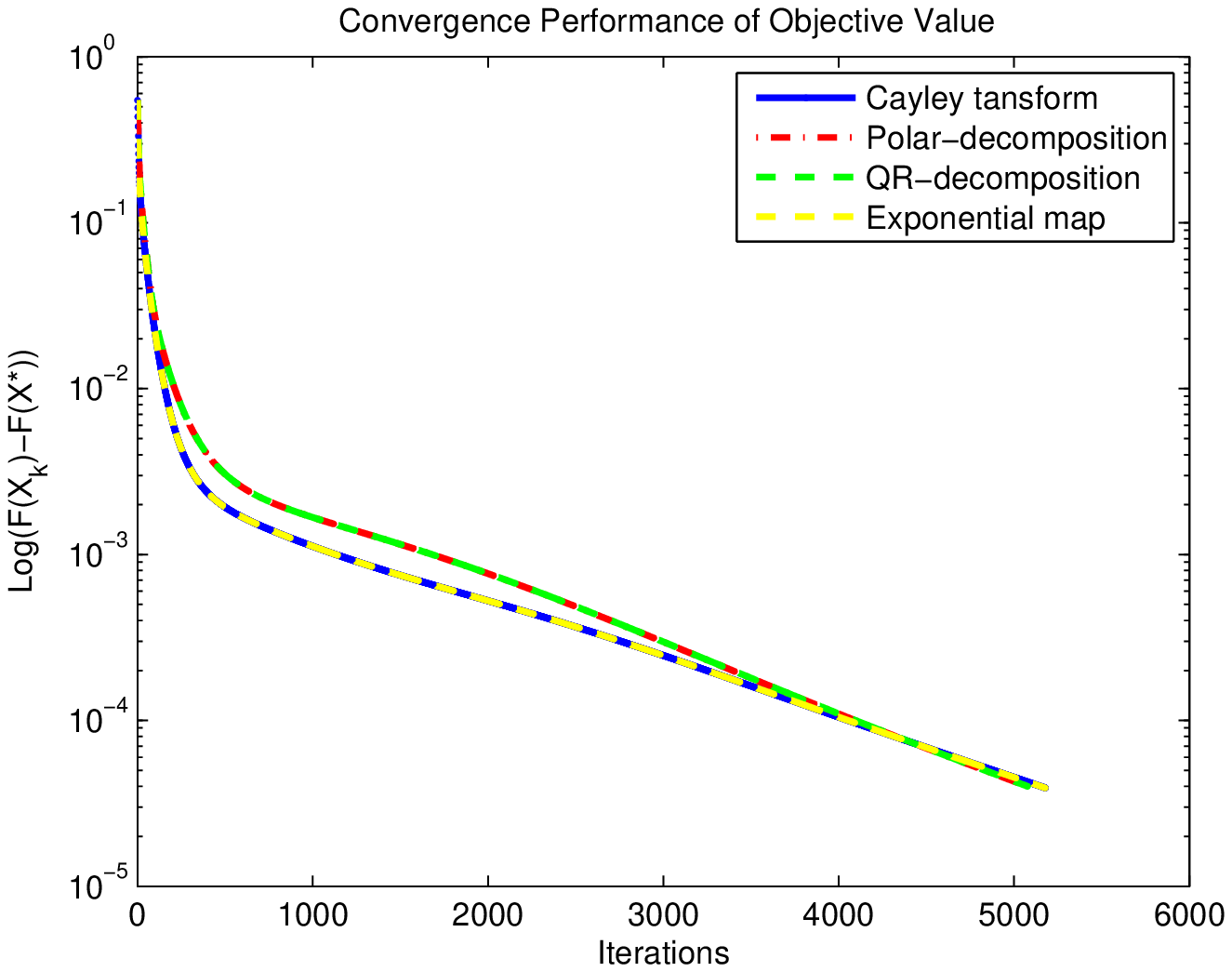}
\label{fig:1}
}
\subfigure[]{
\includegraphics[width=0.3\textwidth]{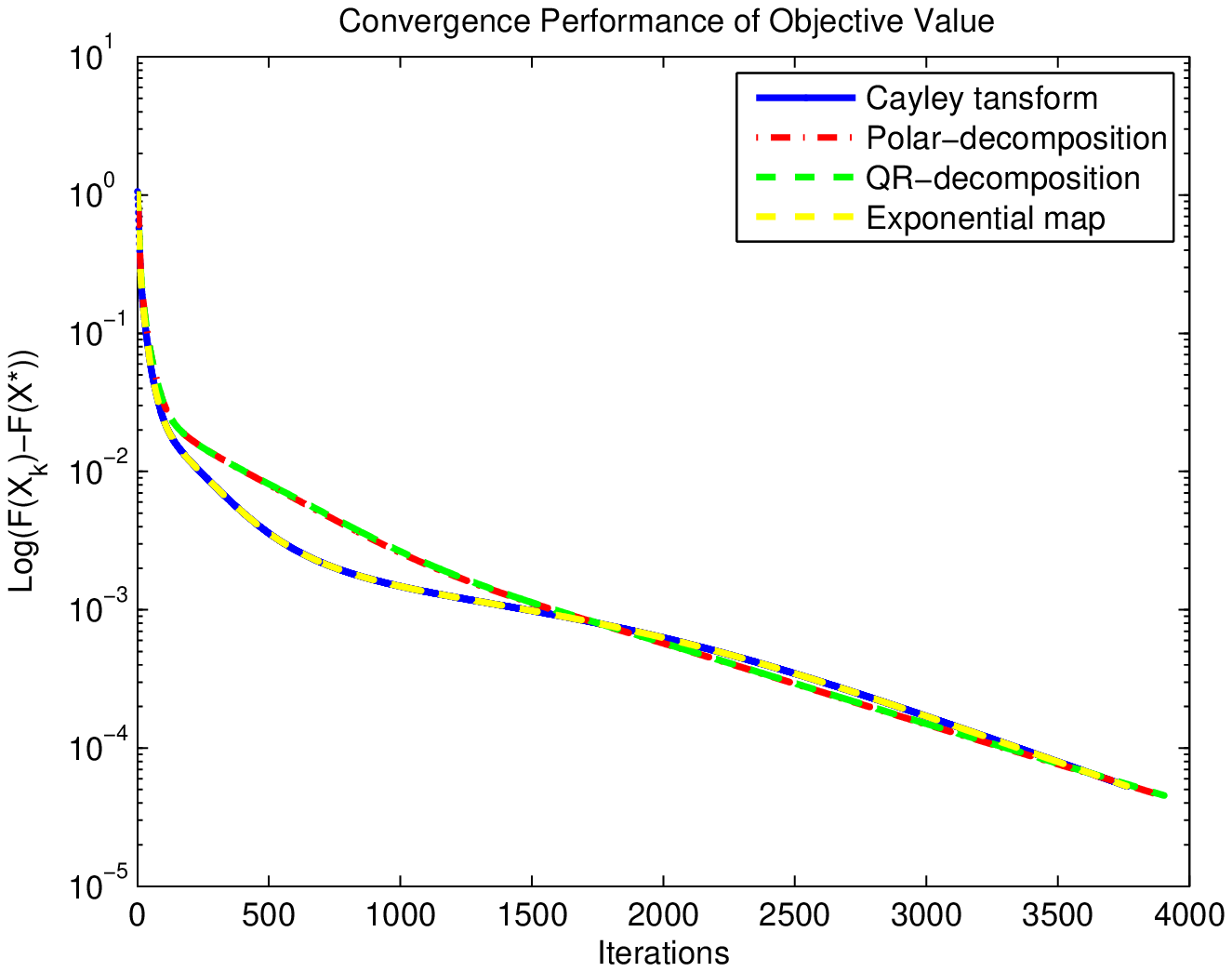}
\label{fig:2}
}
\subfigure[]{
\includegraphics[width=0.3\textwidth]{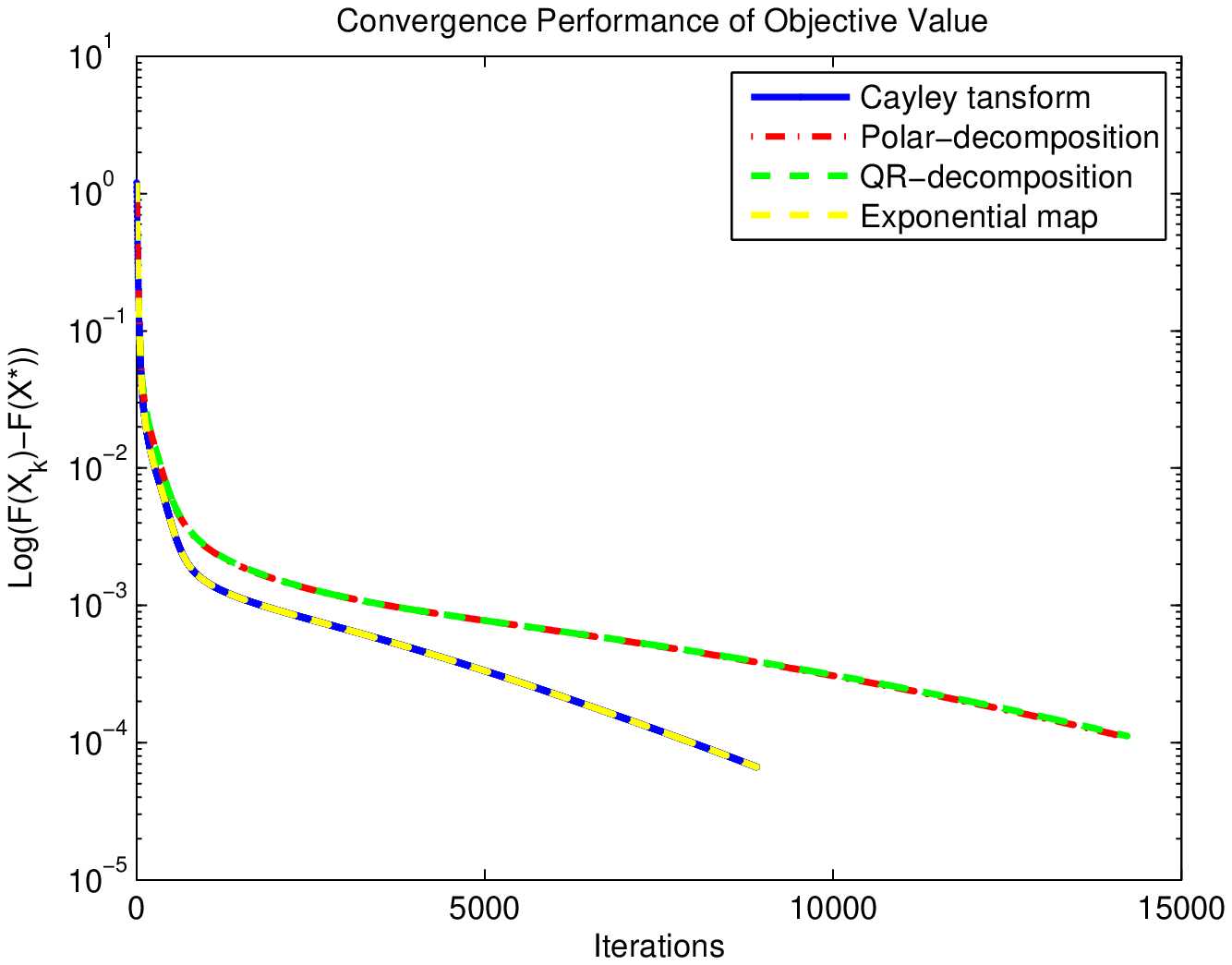}
\label{fig:3}
}
\caption{The Relative ``Thin'' Matrix.}
\label{fig:thin}
\end{figure}

\begin{figure}[h]
\centering
\subfigure[]{
\includegraphics[width=0.3\textwidth]{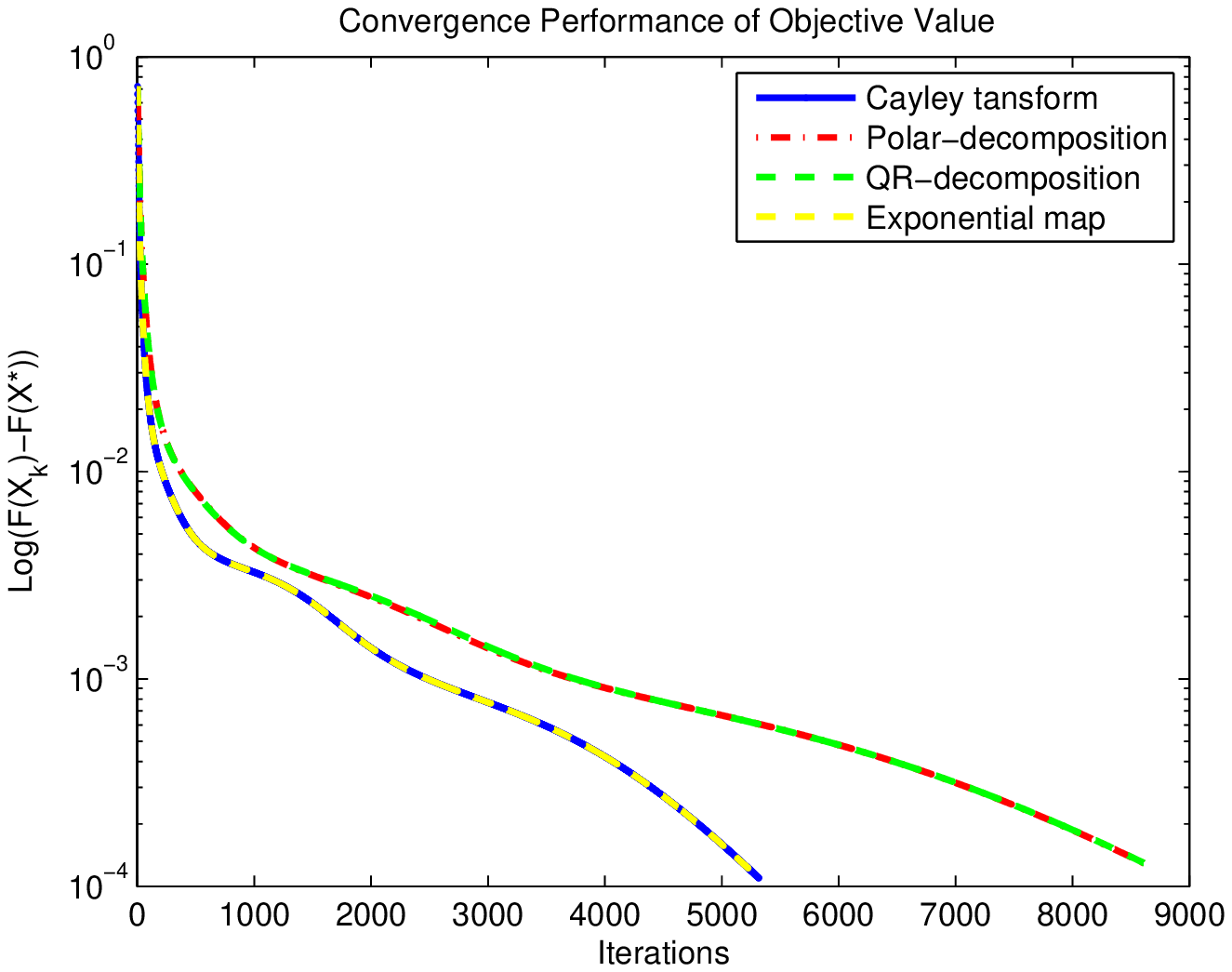}
\label{fig:4}
}
\subfigure[]{
\includegraphics[width=0.3\textwidth]{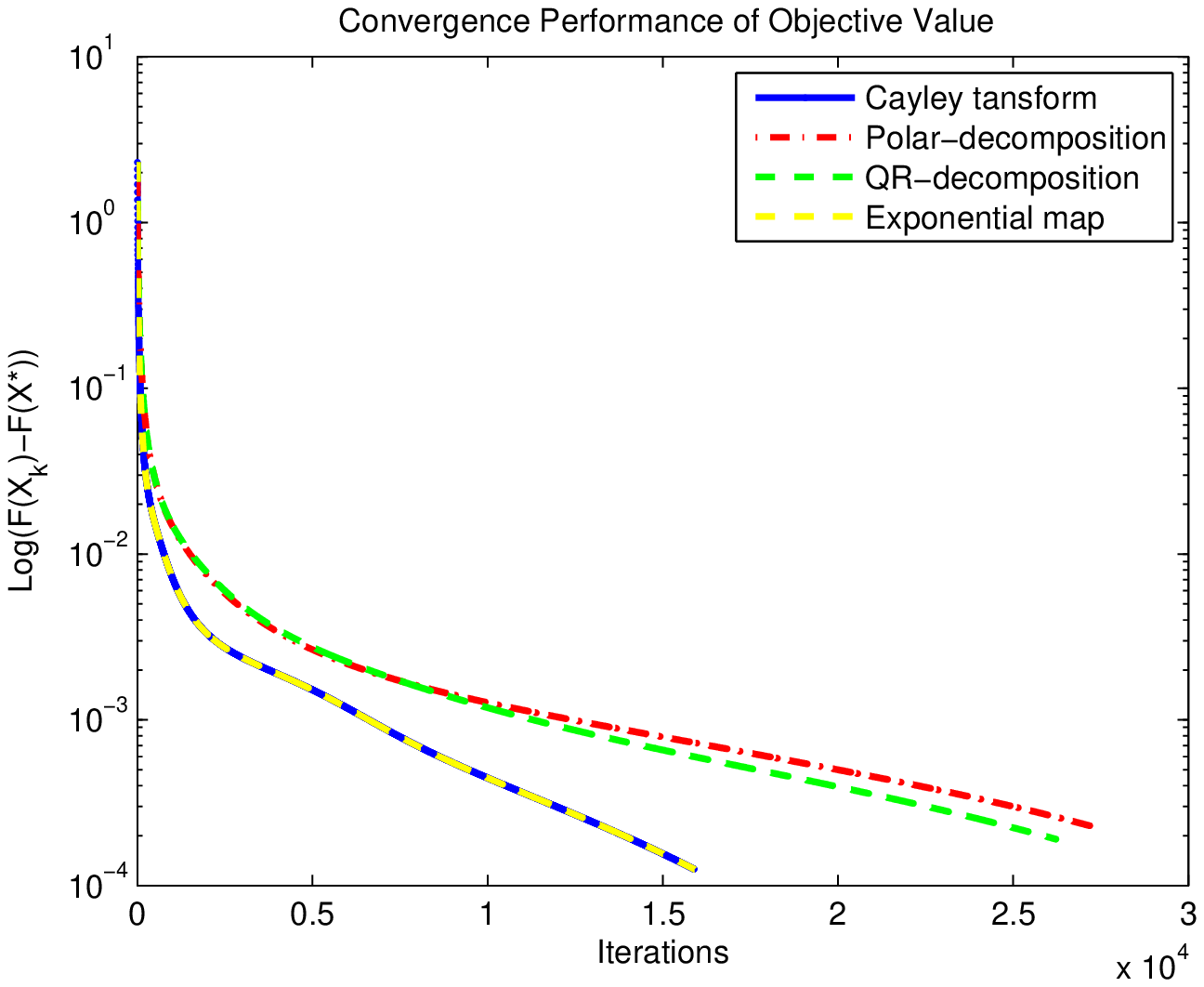}
\label{fig:5}
}
\subfigure[]{
\includegraphics[width=0.3\textwidth]{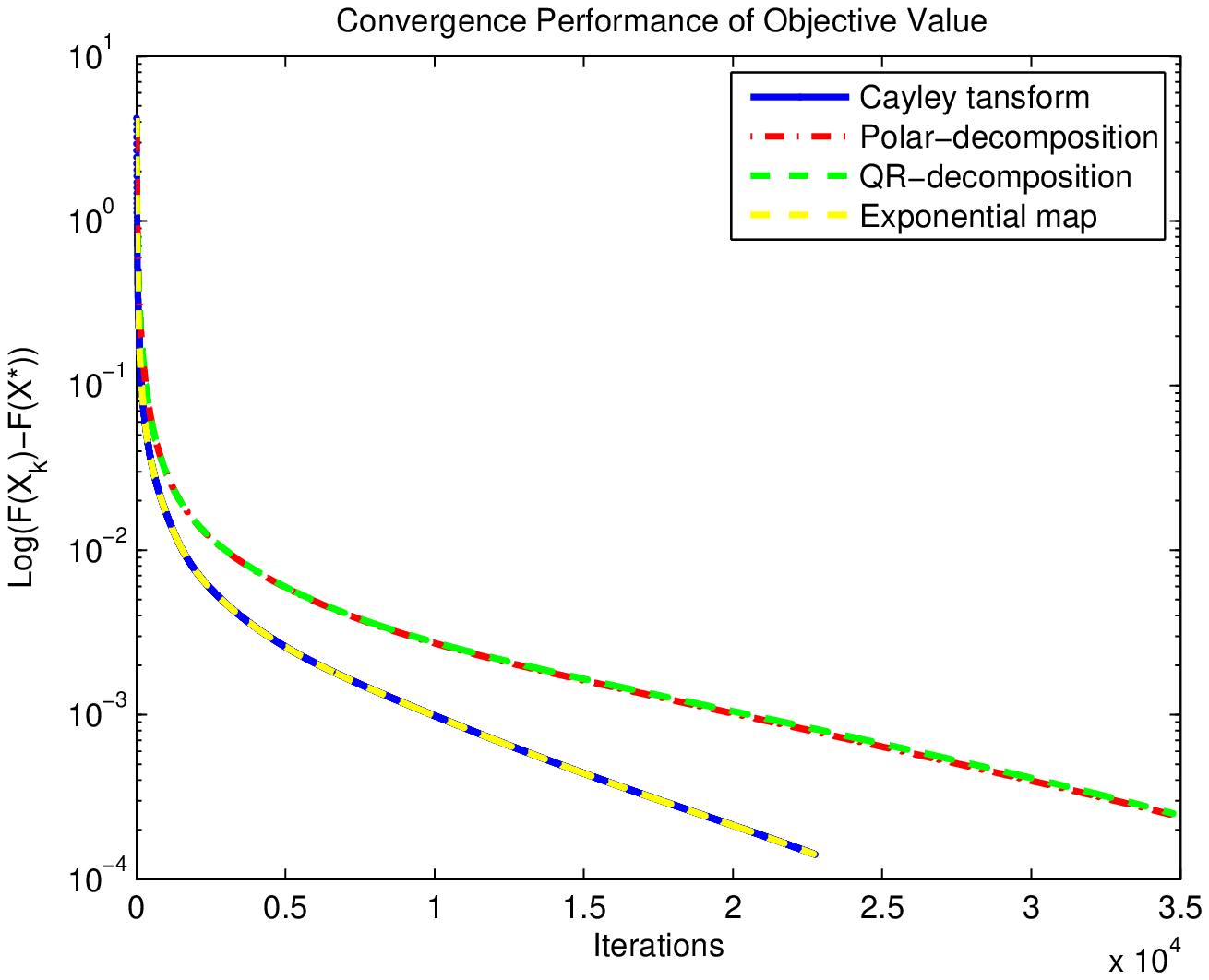}
\label{fig:6}
}
\caption{The Relative ``Fat'' Matrix.}
\label{fig:fat}
\end{figure}

\section{Conclusion}
In this paper, we gave an explicit estimate of the exponent in a {\L}ojasiewicz inequality for the (non-convex) set of critical points of Problem~\eqref{realproblem}.  Such an estimate was obtained by establishing a local error bound for the aforementioned set of critical points.  Together with known arguments, our result implies the linear convergence of a large class of line-search methods on the Stiefel manifold.  An interesting future direction would be to extend our techniques to analyze the convergence rates of first-order methods for solving structured non-convex optimization problems.

\section*{Appendix}
\appendix
\section{Proof of Proposition~\ref{prop:asy-safe}}
The proof of Proposition~\ref{prop:asy-safe} is based on the following lemma:
\begin{lemma}\label{limit}
The Armijo points $\{\alpha_k\}_{k\ge0}$ satisfy $\lim_{k\rightarrow 0}\alpha_k\left\|D_{\rho}(X_k)\right\|_F=0$.
\end{lemma}
\begin{proof}
The Armijo point exists in each step, which guarantees a sufficient decrease. We add all the decrease together and the sum must be finite, since there is a lower bound on the function value; i.e.
\begin{equation*}
\sum_{k=0}^{+\infty} c\alpha_k \|D_{\rho}(X_k)\|_F^2<+\infty,
\end{equation*}
which implies that
 \begin{equation*}
\lim_{k\rightarrow0} \alpha_k \|D_{\rho}(X_k)\|_F^2=0.
\end{equation*}
Here, all the Armijo points have an upper bound $\gamma$.  Thus,
 \begin{equation*}
\lim_{k\rightarrow 0} \alpha_k^2 \|D_{\rho}(X_k)\|^2\leq\lim_{k\rightarrow 0} \gamma \alpha_k \|D_{\rho}(X_k)\|^2 =0.
\end{equation*}
Thus, we have
\begin{equation*}
\lim_{k\rightarrow 0}\alpha_k\|D_{\rho}(X_k)\|=0,
\end{equation*}
as desired.
\end{proof}
{\it Proof of Proposition~\ref{prop:asy-safe}.} By construction of the algorithm, with $V_k= -\alpha_kD_{\rho}(X_k)$, we have
\begin{eqnarray*}
X_{k+1}-X_k=R(X_k, V_k)-X_k=R(X_k, V_k)-(X_k+V_k)-V_k,
\end{eqnarray*}
which implies that
\begin{equation*}
\|X_{k+1}-X_k\|_F\geq \|V_k\|_F-\|R(X_k, V_k)-(X_k+V_k)\|_F.
\end{equation*}
We divide by $\|V_k\|_F$ on both sides to obtain
\begin{equation*}
\frac{\|X_{k+1}-X_k\|_F}{\|V_k\|_F} \geq 1-\frac{\|R(X_k, V_k)-(X_k+V_k)\|_F}{\|V_k\|_F}.
\end{equation*}
It follows that
\begin{equation*}
\lim_{k\rightarrow\infty}\frac{\|X_{k+1}-X_k\|_F}{\|V_k\|_F} \geq 1-\lim_{\|V_k\|_F\rightarrow 0}\frac{\|R(X_k, V_k)-(X_k+V_k)\|_F}{\|V_k\|_F}.
\end{equation*}
According to the definition of smooth retraction~\eqref{retraction}, the last term is equal to 0. Thus,
\begin{equation*}
\lim_{k\rightarrow\infty}\frac{\|X_{k+1}-X_k\|_F}{\|V_k\|_F} \geq 1.
\end{equation*}
Therefore, there exists a large enough $k$ to make sure~\eqref{eq:asy-safe} hold if we choose $1/2$. \qed

\section{Proof of Proposition~\ref{the:limit}}
Let $X\in\mathcal{X}$ be arbitrary.  Using~\eqref{eq:opt-set-sim} and the fact that $X^TX=I_n$, we have $X^TAXB = BX^TAX$.  Since both $X^TAX$ and $B$ are symmetric, this implies that $X^TAX$ and $B$ are simultaneously diagonalizable.  In particular, there exist orthogonal matrices $Q_j \in \mathcal{O}^{t_j-t_{j-1}}$ and diagonal matrices $\Sigma_j \in \mathcal{S}^{t_j-t_{j-1}}$, where $j=1,\ldots,n_B$, such that the columns of ${\rm BlkDiag}(Q_1,\ldots,Q_{n_B})$ are the eigenvectors of $B$, and that
\begin{equation} \label{eq:XAX}
X^TAX = {\rm BlkDiag}\left( Q_1^T\Sigma_1Q_1, \ldots, Q_{n_B}^T\Sigma_{n_B}Q_{n_B} \right).
\end{equation}
Now, using~\eqref{eq:opt-set-sim} again, we have $\left( AX-XX^TAX \right)B=\bz$.  Since $B$ has full rank and hence invertible, this yields $AX = XX^TAX$.  Upon letting $Y=X \cdot {\rm BlkDiag}\left( Q_1^T,\ldots,Q_{n_B}^T \right) \in {\rm St}(m,n)$ and using~\eqref{eq:XAX}, we obtain $AY = Y \cdot {\rm BlkDiag}(\Sigma_1,\ldots,\Sigma_{n_B})$.  As $\Sigma_1,\ldots,\Sigma_{n_B}$ are diagonal, this implies that each of the $n$ columns of $Y$ is an eigenvector of $A$.  To see that $X$ can be expressed in the form given on the right-hand side of~\eqref{eq:opt-set-char}, it remains to note that $A$ has $m$ eigenvectors in total, and that any set of $m$ eigenvectors of $A$ can be expressed as ${\rm BlkDiag}(P_1,\ldots,P_{n_A})$ for some $P_i \in \mathcal{O}^{s_i-s_{i-1}}$, where $i=1,\ldots,n_A$.

The converse is rather easy to verify.  Hence, the proof is completed.

\section{Proof of Proposition~\ref{prop:sq-dist-bd}}
Using~\eqref{dist} and~\eqref{eq:B-decomp}, it can be verified that
\begin{align*}
{\rm dist}^2(X,\mathcal{X}_h) &= \left\| \bar{X} - E(h) \right\|_F^2 \\
&= \min\left\{ \left. \left\| \bar{X} - E(h) \cdot {\rm BlkDiag}(Q_1,\ldots,Q_{n_B}) \right\|_F^2 \,\right|\, Q_j \in \mathcal{O}^{t_j-t_{j-1}} \mbox{ for } j=1,\ldots,n_B \right\} \\
&= \sum_{j=1}^{n_B} \min\left\{ \left. \left\| \bar{X}_j - \bar{E}_j(h)Q_j \right\|_F^2 \,\right|\, Q_j \in \mathcal{O}^{t_j-t_{j-1}} \right\}. 
\end{align*}
From the definitions of $E(h)$ in~\eqref{eq:Eh} and $\bar{E}_j(h)$ in~\eqref{eq:B-decomp}, we see that up to a rearrangement of the rows, $\bar{E}_j(h)$ takes the form $\bar{E}_j(h) = \begin{bmatrix} I_{t_j-t_{j-1}} \\ \bz \end{bmatrix}$.  Thus, to obtain the desired bound on ${\rm dist}^2(X,\mathcal{X}_h)$, it remains to prove the following:
\begin{lemma} \label{theorem_left}
Let $S = \begin{bmatrix} S_1\\S_2 \end{bmatrix} \in {\rm St}(p,q)$ be given, with $S_1 \in \R^{q\times q}$ and $S_2 \in \R^{(p-q)\times q}$.  Consider the following problem:
$$ v^* = \min\left\{ \left. \left\| S - \begin{bmatrix} I_q\\\bz \end{bmatrix} X \right\|_F^2 \,\right|\, X \in \mathcal{O}^q \right\}. $$
Suppose that $v^*<1$.  Then, we have $v^* = \Theta\left( \|S_2\|_F^2 \right)$.
\end{lemma}
\begin{proof}
Since
$$ \left\| S - \begin{bmatrix} I_q\\\bz \end{bmatrix} X \right\|_F^2 = \|S_1-X\|_F^2 + \|S_2\|_F^2, $$
it suffices to consider the problem
\begin{equation} \label{eq:aux-procrust}
\min\left\{ \|S_1 - X\|_F^2 \mid X \in \mathcal{O}^q \right\}.
\end{equation}
Problem~\eqref{eq:aux-procrust} is an instance of the orthogonal Procrustes problem, whose optimal solution is given by $X^*=UV^T$, where $S_1=U\Sigma V^T$ is the singular value decomposition of $S_1$~\cite{S66}.  It follows that
$$ v^* = \|\Sigma-I_q\|_F^2 + \|S_2\|_F^2. $$
Now, since $S\in{\rm St}(p,q)$, we have $S^TS = S_1^TS_1 + S_2^TS_2 = I_q$, or equivalently,
$$ \Sigma^2 + V^TS_2^TS_2V = I_q. $$
This implies that $\bz \preceq \Sigma \preceq I_q$ and
$$
I_q - \Sigma = (I_q+\Sigma)^{-1} \left( V^TS_2^TS_2V \right).
$$
It follows that
$$ \frac{1}{4}\|S_2\|_F^4 + \|S_2\|_F^2 \le v^* \le \|S_2\|_F^4 + \|S_2\|_F^2. $$
This, together with the fact that $\|S_2\|_F^2 \le v^* < 1$, yields $v^* = \Theta\left( \|S_2\|_F^2 \right)$, as desired.
\end{proof}

\section{Proof of Proposition~\ref{prop:grad-prelim-bd}}
Recall that 
$$ P^* = {\rm BlkDiag}\left( P_1^*,\ldots,P_{n_A}^* \right) \in \mathcal{O}^m, \, Q^* = {\rm BlkDiag}\left( Q_1^*,\ldots,Q_{n_B}^* \right) \in \mathcal{O}^n, \, \bar{X} = (P^*)^TX(Q^*)^T. $$
Upon observing that $AP^*=P^*A$, $BQ^*=Q^*B$, $B={\rm BlkDiag}\left( b_{t_1}I_{t_1-t_0},\ldots,b_{t_{n_B}}I_{t_{n_B}-t_{n_B-1}} \right)$ and using~\eqref{eq:B-decomp}, we compute
\begin{align}
\left\| AXB - XBX^TAX \right\|_F^2 &= \left\| AP^*\bar{X}Q^*B - P^*\bar{X}Q^*B(Q^*)^T\bar{X}^T(P^*)^TAP^*\bar{X}Q^* \right\|_F^2 \nonumber \\
&= \left\| P^*\left( A\bar{X}B-\bar{X}B\bar{X}^TA\bar{X} \right)Q^* \right\|_F^2 \nonumber \\
&= \left\| A\bar{X}B-\bar{X}B\bar{X}^TA\bar{X} \right\|_F^2 \nonumber \\
&= \sum_{j=1}^{n_B} \left\| b_{t_j}A\bar{X}_j - \sum_{k=1}^{n_B} b_{t_k}\bar{X}_k\left( \bar{X}_k^TA\bar{X}_j \right) \right\|_F^2. \label{eq:grad-express}
\end{align}
Now, observe that the columns of $\bar{X}$ are orthonormal and span an $n$-dimensional subspace $\mathcal{L}$.  In particular, for $j=1,\ldots,n_B$, each column of $A\bar{X}_j$ can be decomposed as $u+v$, where $u$ is a linear combination of the columns of $\bar{X}$ and $v\in\mathcal{L}^\perp$, the orthogonal complement of $\mathcal{L}$.  In view of the structure of $\bar{X}$ in~\eqref{eq:B-decomp}, this leads to
$$ A\bar{X}_j = \sum_{k=1}^{n_B} \bar{X}_k\left( \bar{X}_k^TA\bar{X}_j \right) + T_j, $$
where $T_j \in\R^{m\times(t_j-t_{j-1})}$ is formed by projecting the columns of $A\bar{X}_j$ onto $\mathcal{L}^\perp$.  Hence,
\begin{align}
 \left\| b_{t_j}A\bar{X}_j - \sum_{k=1}^{n_B} b_{t_k}\bar{X}_k\left( \bar{X}_k^TA\bar{X}_j \right) \right\|_F^2 &= \sum_{k\not=j} (b_{t_j}-b_{t_k})^2 \left\| \bar{X}_k \left( \bar{X}_k^TA\bar{X}_j \right) \right\|_F^2 + b_{t_j}^2 \|T_j\|_F^2 \nonumber \\
&= \Omega\left( \sum_{k\not=j} \left\| \bar{X}_k \left( \bar{X}_k^TA\bar{X}_j \right) \right\|_F^2 + \|T_j\|_F^2 \right) \label{eq:grad-lb} \\
&= \Omega\left( \left\| A\bar{X}_j - \bar{X}_j\bar{X}_j^TA\bar{X}_j \right\|_F^2 \right), \nonumber
\end{align}
where~\eqref{eq:grad-lb} follows from the fact that $b_{t_j}\not=b_{t_k}$ whenever $j\not=k$ and $b_{t_j}\not=0$ since $B$ is assumed to have full rank.  By combining the above with~\eqref{eq:grad-express}, the proof is completed.

\section{Proof of Proposition~\ref{prop:grad-term-bd}}
Consider a fixed $j\in\{1,\ldots,n_B\}$.  Let $\bar{x}_k$ be the $k$-th column of $\bar{X}_j$ and $(\bar{x}_k)_\alpha$ be the $\alpha$-th entry of $\bar{x}_k$, where $k=1,\ldots,t_j-t_{j-1}$ and $\alpha=1,\ldots,m$.
Since ${\rm dist}(X,\mathcal{X}_h) = \left\| \bar{X}-E(h) \right\|_F \le \delta$, using the definition of $E(h)$ in~\eqref{eq:Eh}, we have
$$ 
(\bar{x}_k)_\alpha = \left\{
\begin{array}{c@{\quad}l}
1+O(\delta) &\mbox{if } \alpha=\pi(k), \\
O(\delta) & \mbox{otherwise},
\end{array}
\right.
$$ 
where $\pi(k)$ is the coordinate of the $k$-th column of $\bar{E}_j(h)$ that equals 1.  Since $\pi(k)\not=\pi(\ell)$ whenever $k\not=\ell$, it follows that
$$
\bar{x}_k^TA\bar{x}_{\ell} = \left\{
\begin{array}{c@{\quad}l}
a_{\pi(k)} + O(\delta) & \mbox{if } k=\ell, \\
O(\delta) & \mbox{otherwise}.
\end{array}
\right.
$$
Now, let $\Delta_k$ be the $k$-th column of $A\bar{X}_j-\bar{X}_j\bar{X}_j^TA\bar{X}_j$, where $k=1,\ldots,t_j-t_{j-1}$.  Then,
$$
\Delta_k = A\bar{x}_k - \sum_{\ell=1}^{t_j-t_{j-1}} \bar{x}_\ell \left( \bar{x}_\ell^TA\bar{x}_k \right) = \left( A - a_{\pi(k)}I_m \right)\bar{x}_k - O(\delta) \cdot \left( \sum_{\ell=1}^{t_j-t_{j-1}} \bar{x}_\ell \right).
$$
Let $\Pi_{\mathcal{I}_j}$ be the projector onto the coordinates in $\mathcal{I}_j$.  By Proposition~\ref{prop:sq-dist-bd} and the assumption that ${\rm dist}(X,\mathcal{X}_h) \le \delta$, we have
$$ \sum_{\ell=1}^{t_j-t_{j-1}} \left\| \Pi_{\mathcal{I}_j}(\bar{x}_\ell) \right\|_2^2 = \sum_{k\in\mathcal{I}_j} \left\| \left[ \bar{X}_j \right]_k \right\|_2^2 = O(\delta). $$
Hence,
\begin{align}
\left\| \Pi_{\mathcal{I}_j}(\Delta_k) \right\|_2 &\ge \left\| \Pi_{\mathcal{I}_j} \left(\left( A - a_{\pi(k)}I_m \right)\bar{x}_k\right) \right\|_2 - O(\delta) \cdot \left( \sum_{\ell=1}^{t_j-t_{j-1}} \left\| \Pi_{\mathcal{I}_j}(\bar{x}_\ell) \right\|_2 \right) \nonumber \\
&\ge \left\| \Pi_{\mathcal{I}_j} \left(\left( A - a_{\pi(k)}I_m \right)\bar{x}_k\right) \right\|_2 - O(\delta^2). \label{eq:proj-lb}
\end{align}
Let $i'\in\{0,1,\ldots,n_A-1\}$ be such that $s_{i'}+1 \le \pi(k) \le s_{i'+1}$.  Then, we have
\begin{align}
\left\| \Pi_{\mathcal{I}_j} \left(\left( A - a_{\pi(k)}I_m \right)\bar{x}_k\right) \right\|_2^2 &= \sum_{i\not=i'} \sum_{\alpha\in\mathcal{I}_j \cap \{s_i+1,\ldots,s_{i+1}\}} \left(\left( a_{s_i+1} - a_{\pi(k)} \right)(\bar{x}_k)_\alpha \right)^2 \nonumber \\
&= \sum_{i\not=i'} \sum_{\alpha\in\mathcal{I}_j \cap \{s_i+1,\ldots,s_{i+1}\}} \Omega\left( (\bar{x}_k)_\alpha^2 \right) \nonumber \\
&= \Omega\left( \left\| \Pi_{\mathcal{I}_j} \left( \bar{x}_k\right) \right\|_2^2 \right) - O\left( \left\| \Pi_{\mathcal{I}_j \cap \{s_{i'}+1,\ldots,s_{i'+1}\}} (\bar{x}_k) \right\|_2^2 \right). \label{eq:proj-lb-2}
\end{align}
To bound the term $\left\| \Pi_{\mathcal{I}_j \cap \{s_{i'}+1,\ldots,s_{i'+1}\}} (\bar{x}_k) \right\|_2^2$, we proceed as follows.  Let $Y = X(Q^*)^T \in {\rm St}(m,n)$ and decompose it as
$$
Y = \begin{bmatrix} Y_{11} & \cdots & Y_{1n_A} \\ \vdots & \ddots & \vdots \\ Y_{n_A1} & \cdots & Y_{n_An_A} \end{bmatrix},
$$
where $Y_{ii} \in \R^{(s_i-s_{i-1})\times h_i}$, for $i=1,\ldots,n_A$.  Observe that
\begin{align}
{\rm dist}^2(X,\mathcal{X}_h) &= \min\left\{ \left\| Y - {\rm BlkDiag}(P_1,\ldots,P_{n_A}) \cdot E(h) \right\|_F^2 \mid P_i \in \mathcal{O}^{s_i-s_{i-1}} \mbox{ for }i=1,\ldots,n_A \right\} \nonumber \\
&= \sum_{1\le i\not=j \le n_A} \|Y_{ij}\|_F^2 + \sum_{i=1}^{n_A} \min\left\{ \left. \left\| Y_{ii} - P_i \begin{bmatrix} I_{h_i} \\ \bz \end{bmatrix} \right\|_F^2 \,\right|\, P_i\in\mathcal{O}^{s_i-s_{i-1}} \right\}. \label{eq:subp}
\end{align}
The following lemma establishes a bound on the second term in~\eqref{eq:subp}:
\begin{lemma} \label{prop:subp-opt}
For $i=1,\ldots,n_A$, let
\begin{equation} \label{eq:Y-procrust}
v_i^* = \min\left\{ \left. \left\| Y_{ii} - P_i \begin{bmatrix} I_{h_i} \\ \bz \end{bmatrix} \right\|_F^2 \,\right|\, P_i \in \mathcal{O}^{s_i-s_{i-1}} \right\}.
\end{equation}
Then, we have
$$ v_i^*=\Theta\left( \left\| \sum_{j\not=i} Y_{ji}^TY_{ji} \right\|_F^2 \right). $$
\end{lemma}
Let us defer the proof of Lemma~\ref{prop:subp-opt} to the end of this section.  Together with~\eqref{eq:subp}, Lemma~\ref{prop:subp-opt} implies that
$$ {\rm dist}^2(X,\mathcal{X}_h) = \sum_{1\le i\not=j \le n_A} \|Y_{ij}\|_F^2 + \sum_{i=1}^{n_A} \Theta\left( \left\| \sum_{j\not=i} Y_{ji}^TY_{ji} \right\|_F^2 \right). $$
Since ${\rm dist}(X,\mathcal{X}_h) \le \delta$ for some $\delta\in(0,1)$, we have $\sum_{1\le i\not=j\le n_A} \|Y_{ij}\|_F^2 = O(\delta^2)$.  This implies that
$$ v_i^* = O\left( \left( \sum_{j\not=i} \|Y_{ji}\|_F^2 \right)^2 \right) = O(\delta^4) $$
for $i=1,\ldots,n_A$.

Now, decompose $\bar{X}=(P^*)^TY$ as
$$ \begin{bmatrix} \bar{X}_{11} & \cdots & \bar{X}_{1n_A} \\ \vdots & \ddots & \vdots \\ \bar{X}_{n_A1} & \cdots & \bar{X}_{n_An_A} \end{bmatrix}, $$
where $\bar{X}_{ii} = (P_i^*)^TY_{ii} \in \R^{(s_i-s_{i-1})\times h_i}$ for $i=1,\ldots,n_A$.  Note that for $i=1,\ldots,n_A$, we have
$$ v_i^* = \left\| \bar{X}_{ii} - \begin{bmatrix} I_{h_i} \\ \bz \end{bmatrix} \right\|_F^2. $$
Moreover, observe that $\Pi_{\mathcal{I}_j \cap \{s_{i'}+1,\ldots,s_{i'+1}\}} (\bar{x}_k)$ is part of $\bar{X}_{i'+1,i'+1}$ and does not intersect the diagonal of the top $h_{i'+1}\times h_{i'+1}$ block of $\bar{X}_{i'+1,i'+1}$.  Thus, by Lemma~\ref{prop:subp-opt},
$$ \left\| \Pi_{\mathcal{I}_j \cap \{s_{i'}+1,\ldots,s_{i'+1}\}} (\bar{x}_k) \right\|_2^2 \le v_{i'+1}^* = O(\delta^4). $$
Together with~\eqref{eq:proj-lb} and~\eqref{eq:proj-lb-2}, this yields
$$ \left\| \Pi_{\mathcal{I}_j}(\Delta_k) \right\|_2^2 \ge \Omega \left( \left\| \Pi_{\mathcal{I}_j} \left( \bar{x}_k\right) \right\|_2^2 \right) - O(\delta^3). $$
It follows that
\begin{align*}
\left\| A\bar{X}_j-\bar{X}_j\bar{X}_j^TA\bar{X}_j \right\|_F^2 &= \sum_{k=1}^{t_j-t_{j-1}} \|(\Delta_k)\|_2^2 \\
&\ge \sum_{k=1}^{t_j-t_{j-1}}  \left\| \Pi_{\mathcal{I}_j}(\Delta_k) \right\|_2^2 \\
&\ge \sum_{k=1}^{t_j-t_{j-1}} \Omega\left( \left\| \Pi_{\mathcal{I}_j} \left( \bar{x}_k\right) \right\|_2^2 \right) - O(\delta^3) \\
&= \sum_{k\in\mathcal{I}_j} \Omega\left( \left\|  \left[\bar{X}_j\right]_k \right\|_2^2 \right) - O(\delta^3).
\end{align*}
Upon summing over $j=1,\ldots,n_B$ and using Proposition~\ref{prop:sq-dist-bd}, we obtain the desired bound.

To complete the proof, it remains to prove Lemma~\ref{prop:subp-opt}.

{\it Proof of Lemma~\ref{prop:subp-opt}.} Consider a fixed $i\in\{1,\ldots,n_A\}$.  Note that Problem~\eqref{eq:Y-procrust} is again an instance of the orthogonal Procrustes problem.  Hence, by the result in~\cite{S66}, an optimal solution to Problem~\eqref{eq:Y-procrust} is given by
$$ P_i^* = H_i \begin{bmatrix} W_i^T & \bz \\ \bz & I_{s_i-s_{i-1}-h_i} \end{bmatrix}, $$
where $Y_{ii}=H_i \begin{bmatrix} \Sigma_i \\ \bz \end{bmatrix} W_i^T$ is a singular value decomposition of $Y_{ii}$ with $H_i \in \mathcal{O}^{s_i-s_{i-1}}$, $W_i \in \mathcal{O}^{h_i}$, and $\Sigma_i \in \mathcal{S}^{h_i}$ being diagonal.  It follows from~\eqref{eq:Y-procrust} that
$$ v_i^* = \left\| Y_{ii} - P_i^* \begin{bmatrix} I_{h_i} \\ \bz \end{bmatrix} \right\|_F^2 = \|\Sigma_i - I_{h_i}\|_F^2. $$
Now, since $Y\in{\rm St}(m,n)$, we have
$$ Y_{ii}^TY_{ii} + \sum_{j\not=i} Y_{ji}^TY_{ji} = W_i\Sigma_i^2W_i^T + \sum_{j\not=i} Y_{ji}^TY_{ji} = I_{h_i}, $$
or equivalently,
$$ \Sigma_i^2 + W_i^T \left( \sum_{j\not=i} Y_{ji}^TY_{ji} \right)W_i = I_{h_i}. $$
By following the arguments in the proof of Lemma~\ref{theorem_left}, we conclude that
$$ \|\Sigma_i-I_{h_i}\|_F^2 = \Theta\left( \left\| \sum_{j\not=i} Y_{ji}^TY_{ji} \right\|_F^2 \right), $$
as desired. \hfill \qed

\bibliography{sdpbib}
\bibliographystyle{abbrv}

\end{document}